\documentclass[11pt,reqno]{amsart}
\usepackage{amsmath}
\usepackage{cases}
\usepackage{mathrsfs}
\usepackage{bbm}
\usepackage{amssymb}
\usepackage{amscd}
\usepackage{amsfonts,latexsym,amsmath,amsthm,amsxtra,mathdots,amssymb,latexsym,mathabx}
\usepackage[all,cmtip]{xy}
\RequirePackage{amsmath} \RequirePackage{amssymb}
\usepackage{color}
\usepackage{colordvi}
\usepackage{multicol}
\usepackage{hyperref}
\usepackage{mathtools}
\usepackage[margin=1in]{geometry}
\usepackage{xcolor}
\hypersetup{
	colorlinks,
	linkcolor={red!50!black},
	citecolor={blue!50!black},
	urlcolor={blue!80!black}
}
\usepackage{cite}

\newcommand{\bea}{\begin{eqnarray}}
	\newcommand{\eea}{\end{eqnarray}}
\newcommand{\bna}{\begin{eqnarray*}}
	\newcommand{\ena}{\end{eqnarray*}}

\numberwithin{equation}{section}

\theoremstyle{plain}
\newtheorem*{Theorem A}{Theorem A}
\newtheorem*{Theorem B}{Theorem B}
\newtheorem{problem}{Problem}[section]
\newtheorem{lemma}{Lemma}[section]
\newtheorem{theorem}{Theorem}[section]

\theoremstyle{definition}
\newtheorem{definition}[lemma]{Definition}

\allowdisplaybreaks[3]

\begin{document}
	
	\title
	[{On a problem of Pongsriiam on the sum of divisors, II}] {On a problem of Pongsriiam on the sum of divisors, II}
	
	\author
	[R.-J. Wang] {Rui-Jing Wang}
	
	\address{(Rui-Jing Wang) School of Mathematical Sciences,  Jiangsu Second Normal University, Nanjing 210013, People's Republic of China}
	\email{\tt wangruijing271@163.com}
	
	%\address{(Yu--Chen Sun) Department of Mathematics and Statistics, University of Turku, Turku 20014 , Finland}
	%\email{\tt yucsun@utu.fi}
	
	\subjclass[2010]{11A25, 11N25}
	
	\keywords{arithmetic functions, monotonicity,
		sum of divisors function.}

	\begin{abstract}
		For any positive integer $n$, let $\sigma (n)$ be the sum of all positive divisors of $n.$ In this paper, it is proved that the set of positive integers $ n $ for
		which $ \sigma(30n+1)\geq \sigma(30n) $ has a density less than $ 0.0371813, $ which answers a recent problem of Pongsriiam in part.
	\end{abstract}
	
	\maketitle
	\section{Introduction}
	For any positive integer $n$, let $\sigma (n)$ and $ \phi(n) $ be the sum of divisors function and the Euler totient function, respectively.  In this paper, we always assume that $ x $ denotes a real number, $ a,b,k,l,m,n $ are positive integers and $ p $ is a prime. In 1973, Jarden\cite[p.65]{Jarden} observed that $ \phi(30n+1)>\phi(30n) $ for all $ n\leq 10^{5}. $ Though the inequality was calculated to be true up to $ 10^9, $ Newman\cite{Newman} proved that there are infinitely many $ n $ such that $ \phi(30n+1)<\phi(30n) $ and Martin\cite{Martin} gave the smallest one.  The result of Newman\cite{Newman} gave another example of the Chebyshev's bias\cite{Chebyshev}. For more relavant research, one may refer to \cite{Deleglise,Erdos1936,Luca,Pollack,Wang}. Motivated by these results, it is natural to consider the parallel problem on the sum of divisors function. Recently, Pongsriiam\cite[Theorem 2.4]{Pongsriiam} 
	found that 
	$\sigma(30n)>\sigma(30n+1)$ for all $ n\leq 10^7 $ and proved that $\sigma(30n)-\sigma(30n+1) $ has infinitely many sign changes. Pongsriiam\cite{Pongsriiam} further considered quantitative versions of his result and posed the following two relevant problems.
	\begin{problem}\label{p1}(\cite[Problem 3.8(i)]{Pongsriiam}~)~
		Is it true that $$\sum_{n\leq K\atop \sigma(30n+1)<\sigma(30n)}1>\sum_{n\leq K\atop \sigma(30n+1)>\sigma(30n)}1$$
		for all $ K\in \mathbb{N}? $
	\end{problem}
	\begin{problem}\label{p2}(\cite[Problem 3.8(ii)]{Pongsriiam}~)~
		Is it true that $$\sum_{n\leq K}\sigma(30n)>\sum_{n\leq K}\sigma(30n+1)$$
		for all $ K\in \mathbb{N}? $
	\end{problem}
	Ding, Pan and Sun\cite{Ding} studied several problems of Pongsriiam. Inspired by their ideas, the author\cite{Wang2024} answered affirmatively the above Problem \ref{p2}. For any set of positive integers $ A, $ let $ A(x)=|\{n\leq x: n\in A\}|. $ We use $ \mathbf{d}A $ to denote the density of set $ A. $ Namely, $$ \mathbf{d}A=\lim_{x\rightarrow\infty}\frac{A(x)}{x}. $$ Kobayashi and Trudgian\cite{Kobayashi2020} proved that the set of positive integers $ n $ for
	which $ \sigma(2n+1)\geq \sigma(2n) $ has a density between $ 0.0539171 $ and $ 0.0549445. $ Following their method, we solve Problem \ref{p1} in part.
	\begin{theorem} \label{thm1} Problem \ref{p1} is true for all sufficiently large $ K. $
	\end{theorem}
	In fact, we prove a slightly
	stronger result.
	\begin{theorem} \label{thm2} Let $ B=\{n\geq 1: \sigma(30n+1)\geq \sigma(30n)\}. $ Then $ B $ has a density and $$\mathbf{d}B<0.0371813.$$ 
	\end{theorem}
	For any positive integer $n,$ let $\sigma_{-1} (n)$ be the sum of the reciprocal of all
	positive divisors of $n.$ It is clear that $ \sigma_{-1}(n)=\sigma(n)/n. $ To give the existence of $ \mathbf{d}B, $ we will prove that the sets $ B $ and $ C $ have equal densities, where
	$$C=\{n\geq 1: \sigma_{-1}(30n+1)\geq \sigma_{-1}(30n)\}.  $$
	\section{Existence of $ \mathbf{d}C $}
	
	\begin{definition}(\cite[p.55]{Shapiro}~)%(文章第55页Introduction中提到)
	\ For a real arithmetic function $ f(n) $ defined on $ \mathbb{Z_{+}}, $ let $$F(\mathscr{S},w)=\lim_{x\rightarrow\infty}\frac{|\{n\leq x:n\in \mathscr{S},\ f(n)\leq w\}|}{|\{n\leq x:n\in \mathscr{S}\}|},$$
	where $w$ is a given real number, $ \mathscr{S}\subseteq \mathbb{Z_{+}}. $
	If the limit in the right hand exists except possibly at points of discontinuity of $ F(\mathscr{S},w), $ then $ F(\mathscr{S},w) $ is called the distribution function of $ f(n) $ relative to $ \mathscr{S} $. This limit will be called the $\mathscr{S}$-density of $\{n: n\in \mathscr{S}, f(n)\leq w\}.$
\end{definition}
\begin{definition}(\cite[Definition 3.1]{Shapiro}~)%(文章第60页定义3.1,$\lambda$型的算术函数)
\ Given a sequence of integer-valued arithmetic functions $ \lambda_{k}(n)\ (k=1,2,\ldots). $ An arithmetic function $ f(n) $ is said to be of type $\{\lambda_{k}\}$ relative to $ \mathscr{S} $ if $ f(\lambda_{k}(n)) $ converges in $\mathscr{S}$-density to $ f(n) $ and each of the $ f(\lambda_{k}(n)) $ has a distribution function relative to $ \mathscr{S}. $ In many applications,
\begin{equation*}
	\lambda_{k}(n)=\prod_{p\mid n\atop p\leq k}p^{v_{p}(n)}.
\end{equation*}
Henceforth, we refer to these as the standard type functions. If $ f(n) $ is of $ \lambda $-type with respect to these (relative to $ \mathscr{S} $), we shall say that $ f(n) $ is of standard type relative to $ \mathscr{S}. $
\end{definition}
\begin{lemma}\label{lem:3}
	If $ f(n) $ has a distribution function relative to $ \mathscr{S}, $ so does $ -f(n). $ Both $ f(n) $ and $ -f(n) $ are of the same $ \lambda $-type.
\end{lemma}
Since the continuity of $ F(\mathscr{S},w) $ at point $ w_{0} $ means the $\mathscr{S}$-density of $\{n: n\in \mathscr{S}, f(n)=w_{0}\}$ is zero, Lemma \ref{lem:3} is trivial.
\begin{lemma}(\cite[p.62]{Shapiro}~)\label{lem:4}%(文章第62页(3.10)和(3.12)式)
\ An additive arithmetic function $ f^{+}(n) $ has a distribution function if and only if the series
\begin{equation*}
	\sum_{|f^{+}(p)|>1}\frac{1}{p},\quad \quad \sum_{|f^{+}(p)|\leq 1}\frac{f^{+}(p)}{p},\quad \quad \sum_{|f^{+}(p)|\leq 1}\frac{(f^{+}(p))^{2}}{p}
\end{equation*}
all converge. Similarly, for multiplicative function $ f^{*}(n), $ sufficient condition is that the series
\begin{equation*}
	\sum_{|f^{*}(p)-1|>\frac{1}{2}}\frac{1}{p},\quad \quad \sum_{|f^{*}(p)-1|\leq \frac{1}{2}}\frac{f^{*}(p)-1}{p},\quad \quad \sum_{|f^{*}(p)-1|\leq \frac{1}{2}}\frac{(f^{*}(p)-1)^{2}}{p}
\end{equation*}
all converge. In the case of $ f^{+}(n), $ continuity of the distribution function is equivalent to the divergence of
\begin{equation*}
	\sum_{f^{+}(p)\neq 0}\frac{1}{p}.
\end{equation*}
For $ f^{*}(n), $ continuity of the distribution function is insured by the divergence of
\begin{equation*}
	\sum_{f^{*}(p)\neq 1}\frac{1}{p}.
\end{equation*}
In both the case of additive and multiplicative arithmetic functions described above, the functions are of standard $\lambda$-type.%(文章紧接着提了这句话但未证明,\ 也未标明出处)
\end{lemma}
\begin{definition}(\cite[Definitions 3.4 and 4.3]{Shapiro}~)%(文章第62页定义3.4和第75页定义4.3 progressive subset和multiplicative subset)
\ A subset $\mathscr{S}$ of positive integers is called progressive subset if and only if the intersection of $ \mathscr{S} $ with any given arithmetic progression $ an+b $ has an $\mathscr{S}$-density.

A progressive subset $\mathscr{S}$ of positive integers is called multiplicative subset if and only if the $\mathscr{S}$-density of the intersection of two arithmetic progressions with relatively prime moduli is the product of their respective $\mathscr{S}$-densities.
\end{definition}
\begin{lemma}(\cite[Theorem 4.4]{Shapiro}~)\label{lem:6}%(文章第81页定理4.4)
\ Let $\mathscr{S}$ be a multiplicative subset of positive integers such that for every prime $ p, $
$$\lim_{\alpha\rightarrow\infty}\rho(p^{\alpha})=0,\quad \quad  \sum_{p}\rho^{2}(p)<\infty,$$
where
$ \rho(a)=\max\limits_b \rho_{a,b}, $ $\rho_{a,b} $ is the $\mathscr{S}$-density of the arithmetic progression $ \{ an+b \} $. Consider linear forms $ a_{i}n+b_{i}, i=1,2,\ldots,s $ such that $ a_{i}b_{j}- a_{j}b_{i}\neq 0\ (i\neq j),\ a_{i}\neq 0,$ and given arithmetic functions $ f_{1},f_{2}\ldots,f_{s} $ of standard type. Then, if the distribution function of $ f_{1} $ is continuous,
\begin{equation*}
	f_{1}(a_{1}n+b_{1})+f_{2}(a_{2}n+b_{2})+\cdots+f_{s}(a_{s}n+b_{s})
\end{equation*}
has a continuous distribution function.
\end{lemma}
Following the definitions and lemmas mentioned above, we can conclude that 
\begin{lemma}\label{lem:7}
	For given arithmetic functions $ f_{1},f_{2}\ldots,f_{s} $ satisfying the conditions of Lemma \ref{lem:4} and linear forms $ a_{i}n+b_{i}, i=1,2,\ldots,s $ satisfying $ a_{i}b_{j}-a_{j}b_{i}\neq 0\ (i\neq j),\ a_{i}\neq 0,$ let
	$$ D=\{n\geq 1:f_{1}(a_{1}n+b_{1})+f_{2}(a_{2}n+b_{2})+\cdots+f_{s}(a_{s}n+b_{s})\leq w\} .$$ Then $ D $ has a density for each given real number $w.$ 
\end{lemma}
By Lemma \ref{lem:3}, each $ f_{i} $ can be replaced by $ -f_{i} $ for $ i=1,2,\ldots,s. $
\begin{proof}
	By the Chinese remainder theorem, $ \mathbb{Z_{+}} $ is a multiplicative subset. For every prime $ p, $
	$$
	\lim_{\alpha\rightarrow\infty}\rho(p^{\alpha})=\lim_{\alpha\rightarrow\infty}\frac{1}{p^{\alpha}}=0,
	\quad \quad  \sum_{p}\rho^{2}(p)=\sum_{p}\frac{1}{p^{2}}<\sum_{n=1}^{\infty}\frac{1}{n^{2}}=\frac{\pi^{2}}{6}.
	$$
	By Lemma \ref{lem:4}, $ f_{1},f_{2}\ldots,f_{s} $ are of standard type and $ f_{1} $ has a continuous distribution function. Thus, by Lemma \ref{lem:6}, $ D $ has a density. 
	
	This completes the proof of Lemma \ref{lem:7}.
\end{proof}

Let $ w =0,\ f(n)=\sigma_{-1}(30n)-\sigma_{-1}(30n+1) $. It follows from $ \sigma_{-1}(p)-1=1/p $ that
$$
	\sum_{|\sigma_{-1}(p)-1|>\frac{1}{2}}\frac{1}{p}=0,
\quad \quad \sum_{|\sigma_{-1}(p)-1|\leq \frac{1}{2}}\frac{\sigma_{-1}(p)-1}{p}=\sum_{p}\frac{1}{p^{2}},
\quad \quad \sum_{|\sigma_{-1}(p)-1|\leq \frac{1}{2}}\frac{(\sigma_{-1}(p)-1)^{2}}{p}=\sum_{p}\frac{1}{p^{3}}
$$
are all converge. But
\begin{equation*}
	\sum_{\sigma_{-1}(p)\neq 1}\frac{1}{p}=\sum_{p}\frac{1}{p}
\end{equation*}
diverges.
Therefore, $ \sigma_{-1} $ satisfies the conditions of Lemma \ref{lem:4}. 
By Lemma \ref{lem:7}, $ C $ has a density.

	\section{Notations and preliminary lemmas}
	
	\begin{lemma}(\cite[p.253]{Erdos1987}~)\label{lemD1}
		For large $ x, $ 
		$$|\left\{1\leq n\leq x: \sigma(n+1)=\sigma(n) \right\}|\leq x/e^{(\log x)^{1/3}}.  
		$$
	\end{lemma}
	\begin{lemma}(\cite[p.119]{Gronwall}~)\label{lemD2}
		We have 
		$$ \limsup_{n\to\infty}\frac{\sigma(n)/n}{\log\log n}=e^{\gamma}, 
		$$
		where $ \gamma $ denotes the Euler constant.
	\end{lemma}
	\begin{lemma}(\cite[(5.10)
		and (5.12)]{Norton}~)\label{lemD3}
		Let $Q(\alpha) = (1 + \alpha) \log(1 + \alpha)-\alpha.  $ Then for large $ x, $ 
		$$|\left\{1\leq n\leq x: \omega(n)\leq (1-\alpha) \log\log x\right\}|\leq x/(\log x)^{Q(-\alpha)},  
		$$	
		where $ \omega(n) $ denotes the number of distinct prime
		factors of $ n $ and $ 0<\alpha<1. $
	\end{lemma}
	Recall that $$ B=\{n\geq 1: \sigma(30n+1)\geq \sigma(30n)\},\quad \quad C=\{n\geq 1: \sigma_{-1}(30n+1)\geq \sigma_{-1}(30n)\}.  $$
	\begin{lemma}\label{lem0}
	$ B $ has a density and
	$\mathbf{d}B=\mathbf{d}C.$	
	\end{lemma}
	\begin{proof}
	Since $$ \frac{\sigma_{-1}(30n+1)}{\sigma_{-1}(30n)}=\frac{\sigma(30n+1)}{\sigma(30n)}\cdot\frac{30n}{30n+1}, $$
	we have $ C\subseteq B. $
	It follows that $$ B\setminus C=\left\{n\geq 1: 0\leq \sigma(30n+1)-\sigma(30n)<\frac{\sigma(30n)}{30n}\right\}. $$  We will prove that the set $ B\setminus C $ has density $ 0. $ Let $$ D_{0}=\left\{n\geq 1: \sigma(30n+1)=\sigma(30n) \right\}, \quad  \quad D_{1}=\left\{n\geq 1: \frac{\sigma(30n)}{30n}>2e^{\gamma}\log\log n\right\}, $$
	$$ D_{2}=\left\{n\geq 1: \omega(30n+1)\leq \frac{1}{2}\log\log n\right\},\quad  \quad D_{3}=\left\{n\geq 1: \omega(30n)\leq\frac{1}{2}\log\log n\right\}, $$
	$$ D_{4}=\left\{n\geq 1: p^{2}\mid 30n\ \text{for\ some}\ p>y\right\},\quad  \quad D_{5}=\left\{n\geq 1: p^{2}\mid 30n+1\ \text{for\ some}\ p>y\right\}. $$
	%$$ D_{6}=\left\{n\geq 1: 2^{\frac{1}{4}\log\log n}\leq 2e^{\gamma}\log\log n\right\} $$
	where $ \gamma $ denotes the Euler constant and $ y $ is a given real number.
	Let $ y=\frac{1}{4}\log\log x. $ Then by the definitions of $ D_{1},\ldots,D_{5}, $ for any positive integer $ n $ satisfying $ n\in (B\setminus C)\setminus (D_{1}\cup D_{2}\cup D_{3}\cup D_{4}\cup D_{5}) $ and $ n\geq \sqrt{x}, $ either $ \sigma(30n+1)=\sigma(30n) $ or 
	$$ 2^{\frac{1}{2}\log\log \sqrt{x}-\pi\left(\frac{1}{4}\log\log x\right)}\leq \sigma(30n+1)-\sigma(30n)<\frac{\sigma(30n)}{30n}\leq 2e^{\gamma}\log\log x. $$
	Therefore, for all sufficiently large $ x,$
	$$ (B\setminus C)(x)\leq \sqrt{x}+\sum_{i=0}^{5}D_{i}(x). $$
	Dividing by $ x $ and taking $ x \rightarrow \infty, $ it follows from Lemmas \ref{lemD1}, \ref{lemD2} and \ref{lemD3} that $\mathbf{d}(B\setminus C)=0.$
	Therefore, $ B $ has a density and
	$\mathbf{d}B=\mathbf{d}C.$	
	
	This completes the proof of Lemma \ref{lem0}.
	\end{proof}

	Let $ y\geq 5. $ Then a number $ n $ is $ y $-smooth if its largest prime divisor $ p $ has $ p \leq y. $ Let $ Y (n) $ be the largest $ y $-smooth divisor of
	$ n. $ In the following discussions, we always assume that $ a$ and $b $ are $ y $-smooth. Let 
	$$
	S(a, b)=\{n\geq 1:Y(30n+1)=a,Y(30n)=b\}.
	$$
	It is clear that the sets $ S(a, b) $ partition $ \mathbb{N_+}. $ $ S(a, b) = \emptyset $ unless $ 30\mid b $ 
	and $ (a, b) = 1. $ We partition $ C $ via $ C(a, b) = C\cap S(a, b). $
	Let $ P=P(y)=\prod_{p\leq y}p,$ 
	$$
	S\left(a, b ; t_1, t_2\right)=\left\{n \in S(a, b):(30n+1) / a \equiv t_1 \hskip-3mm\pmod P,\quad  30n / b \equiv t_2 \hskip-3mm\pmod P\right\},
	$$
	where $ 1\leq t_1\leq P, $ $ 1\leq t_2\leq P $ and $ (t_1t_2,P)=1. $
	\begin{lemma}\label{lem1}
	We have
	$$
	\mathbf{d} S\left(a, b ; t_1, t_2\right)= \begin{cases} \frac{30}{a b P} & P \mid 1-a t_1+b t_2, \\ 0 & P \nmid 1-a t_1+b t_2,\end{cases}
	\quad\quad \mathbf{d} S(a, b)=\frac{30}{a b} \prod_{p \mid a b}\left(1-\frac{1}{p}\right) \prod_{\substack{p \mid P \\ p \nmid a b}}\left(1-\frac{2}{p}\right).
	$$
	Either $ S\left(a, b ; t_1, t_2\right)=\emptyset $ or
	$ n\in S\left(a, b ; t_1, t_2\right) $ has the form
	$$
		30n+1=a\left(t_1+x_0 P \ell\right)+a b P k,\quad \quad 
		30n =b\left(t_2+y_0 P \ell\right)+a b P k,
	$$ 
	where $ (x_0, y_0) $
	is a particular solution for $ ax-by=1, $ $ P \ell=1-a t_1+b t_2 $ and $ k\in \mathbb{Z}. $
	\end{lemma}
	\begin{proof}
		We follow the method of Kobayashi and Trudgian\cite[p.141]{Kobayashi2020} to show that $ S(a, b) $ is a finite and disjoint union of arithmetic progressions. By the definition of $ S\left(a, b ; t_1, t_2\right), $ $$ 30n + 1 = a(t_1+xP), \quad\quad  30n = b(t_2+yP), $$ where $ x,y\in \mathbb{Z_{+}}. $ Hence,
		$$ aPx-bPy=1-at_1+bt_2. $$
		If $ P\nmid 1-at_1+bt_2, $ then $ S\left(a, b ; t_1, t_2\right)=\emptyset. $ If $ P\mid 1-at_1+bt_2, $ write $ P\ell=1-at_1+bt_2, $ then $ ax-by=\ell. $ It follows that $$ x=x_0\ell+kb,\quad \quad  y=y_0\ell+ka, $$
		where $ (x_0, y_0) $
		is a particular solution for $ ax-by=1 $  and $ k\in \mathbb{Z}. $ Therefore,
		$$
		30n+1=a\left(t_1+x_0 P \ell\right)+a b P k,\quad \quad 
		30n =b\left(t_2+y_0 P \ell\right)+a b P k.
		$$
		It follows that
		$$
		\mathbf{d} S\left(a, b ; t_1, t_2\right)= \begin{cases} \frac{30}{a b P} & P \mid 1-a t_1+b t_2, \\ 0 & P \nmid 1-a t_1+b t_2.\end{cases}
		$$
		In order to give $ \mathbf{d} S(a, b), $ we will count the number of pairs $ (t_1,t_2) $ satisfying
		$ P \mid 1-a t_1+b t_2. $ For each prime $ p\mid P, $ if $ p\mid a, $ then
		$ P \mid 1-a t_1+b t_2 $ if and only if $ t_2\equiv -b^{-1} \pmod p, $ so $ t_1 $ is free and $ t_2 $ is completely
		determined modulo $ p. $ Thus, there are $ p-1 $ ordered pairs modulo $ p. $ Likewise, if $ p\mid b, $
		$ P \mid 1-a t_1+b t_2 $ if and only if $ t_1\equiv a^{-1} \pmod p, $ so there are $ p-1 $ ordered pairs
		modulo $ p. $ Finally, if $ p\nmid ab, $ $ P \mid 1-a t_1+b t_2 $ if and only if $ t_2\equiv b^{-1}(at_1-1) \pmod p. $ For
		each $ t_1, $ we fail to get a valid $ t_2 $ only if $ t_1\equiv a^{-1} \pmod p. $ Thus, there are
		$ p-2 $ valid ordered pairs modulo $ p. $ By the Chinese remainder theorem, the number of pairs $ (t_1,t_2) $ satisfying
		$ P \mid 1-a t_1+b t_2 $ is
		$$ \prod_{p \mid a b}\left(p-1\right) \prod_{\substack{p \mid P \\ p \nmid a b}}\left(p-2\right). $$
		Therefore,
		$$
		\mathbf{d} S(a, b)=\frac{30}{a b} \prod_{p \mid a b}\left(1-\frac{1}{p}\right) \prod_{\substack{p \mid P \\ p \nmid a b}}\left(1-\frac{2}{p}\right).
		$$
		
		This completes the proof of Lemma \ref{lem1}.
	\end{proof}
	\begin{lemma}(\cite[Lemma 1]{Kobayashi2020}\ )\label{lem2}
	Let	$ g(n)=(\sigma_{-1}(n))^r $ and let $ h(n) $ be the M\"obius transform of $ g(n). $ Namely, $ h(p^\alpha)=g(p^\alpha)-g(p^{\alpha-1}). $ If $ k $ and $ l $ are given coprime positive integers, $ r\geq 1 $ and $ x\geq 2, $ then
	$$
	\sum_{\substack{n \leq x \\ n \equiv l \hskip-3mm\pmod k}}g(n)= \frac{\Lambda_k(r)}{k}x+O\left(k^{1 / 2}\log k(\log x)^{r}\right),
	$$
	where $$
	\Lambda_k(r)=\prod_{p \nmid k}\left(1+\frac{h(p)}{p}+\frac{h\left(p^2\right)}{p^2}+\cdots\right).
	$$
	\end{lemma}

	\section{Proofs of Theorems \ref{thm1} and \ref{thm2}}

\begin{proof} 

By Lemmas \ref{lem1} and \ref{lem2},
$$
\begin{aligned}
	\sum_{\substack{n \leq x \\
			n \in S\left(a, b ; t_1, t_2\right)}} g(30n+1) & =g(a) \sum_{\substack{m \leq(30x+1) / a \\
			m \equiv t_1+x_0 P \ell \hskip-3mm\pmod {b P}}} g(m) \\
	& =g(a) \Lambda_{b P}(r) \frac{30}{a b P} x+O\left((\log x)^{r}\right)\\
	& =g(a) \Lambda_{P}(r) \frac{30}{a b P} x+O\left((\log x)^{r}\right) .
\end{aligned}
$$
Summing over all pairs $ (t_1,t_2), $ 
$$
\sum_{\substack{n \in S(a, b) \\ n \leq x}} g(30n+1)\sim g(a)\Lambda_{P}(r)\mathbf{d}S(a, b)x, \quad \quad  x \rightarrow \infty. $$
By the definition of $ C, $
$$
\sum_{\substack{n \in S(a, b) \\ n \leq x}} g(30n+1) \geq \sum_{\substack{n \in S(a, b) \\ n \leq x,\ n\in C}} g(30n)+\sum_{\substack{n \in S(a, b) \\ n \leq x,\ n\not\in C}} g(30n+1) .
$$
It is clear that $ g(n)\geq 1, $ $ g(30n)=g(b)g(30n/b) $ and $ g(30n+1)=g(a)g((30n+1)/a). $ By the definitions of $ S(a, b) $ and $ C(a, b), $
$$
\sum_{\substack{n \in S(a, b) \\ n \leq x}} g(30n+1) \geq g(b)|C(a, b) \cap[1, x]|+g(a)(|S(a, b) \cap[1, x]|-|C(a, b) \cap[1, x]|) .
$$
Dividing by $ x $ and taking $ x \rightarrow \infty, $ 
$$ g(a)\Lambda_{P}(r)\mathbf{d}S(a, b)\geq g(b)\mathbf{d}C(a, b)+g(a)\mathbf{d}S(a, b)-g(a)\mathbf{d}C(a, b). $$
Since $ \Lambda_{P}(r)>1 $ and $ r\geq 1, $ for $ (\Lambda_{P}(r))^{1/r}<\sigma_{-1}(b)/ \sigma_{-1}(a),$ we have
$$ \mathbf{d}C(a, b)\leq \frac{(\Lambda_{P}(r)-1)g(a)\mathbf{d}S(a, b)}{g(b)-g(a)}. $$
We take $ r=1. $ . Then by the Euler product formula,
$$
\Lambda_P(1)=\zeta(2) \prod_{p \mid P}\left(1-\frac{1}{p^2}\right).
$$
In calculation, we fix the parameters $ y$ and $ z, $ then run through $ y $-smooth numbers $ a $ and $ b $ satisfying $ ab\leq z, $ $ 30\mid b $ 
and $ (a, b) = 1. $ After summing over every difference between $ \mathbf{d}S(a, b) $ and the above nontrivial bound of $ \mathbf{d}C(a, b) $ for each pair $ (a, b), $  $ 1 $ minus the sum is an upper bound of $ \mathbf{d}C. $  The choice $ y = 157, z = 10^{9} $ yields the value $ 0.0371813. $ 

Theorem \ref{thm2} follows Lemma \ref{lem0} and Theorem \ref{thm1} then follows immediately.

\end{proof}
	
	\bigskip
	
	\section*{Acknowledgments}
	The author would like to thank Yong-Gao Chen and Yuchen Ding for their helpful suggestions. 
	
	This work is supported by the Natural Science Foundation of the Jiangsu Higher Education Institutions of China (Grant No. 24KJD110001).

\end{document}